\documentclass[10pt]{article}
\font\smallit=cmti10

\usepackage{amssymb,latexsym,amsmath,epsfig,amsthm} 
\usepackage{enumitem}   

\makeatletter

\renewcommand\section{\@startsection {section}{1}{\z@}
{-30pt \@plus -1ex \@minus -.2ex}
{2.3ex \@plus.2ex}
{\normalfont\normalsize\bfseries\boldmath}}

\renewcommand\subsection{\@startsection{subsection}{2}{\z@}
{-3.25ex\@plus -1ex \@minus -.2ex}
{1.5ex \@plus .2ex}
{\normalfont\normalsize\bfseries\boldmath}}

\renewcommand{\@seccntformat}[1]{\csname the#1\endcsname. }

\makeatother

\newtheorem{theorem}{Theorem}
\newtheorem{lemma}{Lemma}
\newtheorem{proposition}{Proposition}
\newtheorem{corollary}{Corollary}

\theoremstyle{definition}

\newtheorem{remark}{Remark}
\newtheorem{example}{Example}

\newcommand{\coeff}{\mathrm{coeff}}
\newcommand{\ee}{\mathrm{e}}

\renewcommand{\AA}{\mathtt{A}}
\newcommand{\BB}{\mathtt{B}}

\newcommand{\QQ}{\mathbb{Q}}
\newcommand\nA{\mathcal{A}}

\newcommand{\denom}{\mathrm{denom}}
\newcommand{\lcm}{\mathrm{lcm}}

\begin{document}


\begin{center}
\uppercase{\bf Denominators of coefficients of the Baker--Campbell--Hausdorff series}
\vskip 20pt
{\bf Harald Hofst\"atter}\\ 
{\smallit Reitschachersiedlung 4/6, 7100 Neusiedl am See, Austria}\\
{\tt hofi@harald-hofstaetter.at}\\ 
\end{center}
\vskip 20pt



\centerline{\bf Abstract}

\noindent
For the computation 
of terms
of the Baker--Campbell-Hausdorff series $H = \log(\ee^{\AA}\ee^{\BB})$ 
some a priori knowledge about the denominators of the coefficients of 
the series can be beneficial.
In this paper an explicit formula for the computation of 
 common denominators 
for the rational coefficients of the homogeneous components of the series is derived.
Explicit computations up to degree 30 show that  the common
denominators obtained by this formula are as small as possible, which suggests
that the formula is in a sense optimal. 
The sequence of integers defined by the formula seems to be interesting also from a number-theoretic point of view. There is, e.g., a connection with the
denominators of the Bernoulli numbers and the Bernoulli polynomials.

\pagestyle{myheadings}
\thispagestyle{empty}
\baselineskip=12.875pt
\vskip 30pt

\section{Introduction}
We consider the 
Baker--Campbell--Hausdorff (BCH) series which is formally defined as
the element 
\begin{equation*}
H = \log(\ee^{\AA}\ee^{\BB})
= \sum_{k=1}^\infty\frac{(-1)^{k+1}}{k}\big(\ee^{\AA}\ee^{\BB}-1\big)^k
= \sum_{k=1}^\infty\frac{(-1)^{k+1}}{k}\bigg(\sum_{p+q>0}\frac{1}{p!q!}\AA^p\BB^p\bigg)^k 
\end{equation*}
in the ring $\QQ\langle\langle\AA,\BB\rangle\rangle$  of formal power series 
in the non-commuting variables $\AA$ and $\BB$ with rational coefficients.
A classical result known as the
Baker--Campbell--Hausdorff theorem states that 
$H$
is  a \emph{Lie series}, i.e.,  a sum $H=\sum_{n=1}^{\infty}H_n$
of homogeneous components $H_n$ which 
can be written as linear combinations of 
$\AA$ and $\BB$ and (possibly nested) 
 commutator terms in $\AA$ and $\BB$. For an accessible proof of the BCH theorem see,
 e.g., 
 \cite{Eichler} or \cite{HHproof}.
  
The algorithmic determination of the homogeneous components $H_n$  turns out 
to be quite nontrivial,
especially if the $H_n$ are to  be represented as  linear combinations of 
linearly independent commutators.
The determination of the coefficients $h_w\in\QQ$ 
in the representations 
\begin{equation*}
H_n= \sum_{w\in\{\AA,\BB\}^n} h_w w,\quad n=1,2,\dots
\end{equation*}
can be an
important  first step for this task.
Here, $\{\AA,\BB\}^n$
denotes the finite set of all words $w$ of length (degree) $|w|=n$ over the 
alphabet $\{\AA, \BB\}$.
It would be  beneficial if some a priori information about the denominators
of the coefficients $h_w=\coeff(w,H)$ 
would be available, 
which would allow to compute a common denominator valid for all $h_w\in\{\AA,\BB\}^n$, because it would  then  be possible to compute these
coefficients
using pure integer arithmetic rather than less efficient rational
arithmetic (cf.~\cite{HHfast}). In this context, K.~Goldberg stated
in the penulti\-mate paragraph of \cite{G}:
\begin{quote}
[T]he chief difficulty is that computation with rationals is unavoidable until
some idea of the factorization of the denominators of the coefficients is known.
However for the small degrees, $n\leq 10$, all the denominators for the same
degree $n$ divide the denominator of $(B_{n-1}+B_{n-2})/n!$ and this may be the 
general case.
\end{quote}
Here $B_n$ denote the Bernoulli numbers.
Unfortunately, already for degree $n=11$ the  denominator of the given formula is
not a valid common  denominator for all  coefficients corresponding to  this degree.\footnote{Also for $n\leq 3$ the
given formula is not entirely correct and has to be properly interpreted.}
 Indeed, for the word $w=\AA^8\BB^3\in\{\AA,\BB\}^{11}$
we have\footnote{This value can for example be looked up in the table
given in \cite{NewmanThompson}.}
$\!h_w=1/1247400$ and $(B_{10}+B_9)/11!=1/526901760$ but 
 $526901760/1247400=2212/5\not\in\mathbb{Z}$.
However, 
a valid such common denominator
 is given by the following theorem, which is the main result of the present paper.
\begin{theorem}\label{Thm:MainTheorem}
For $n\geq 1$ define
\begin{equation}\label{eq:d_n}
d_n = \prod_{p\ \mathrm{prime}, \ p<n}p^{\max\{t:\  p^t\leq s_p(n)\}}, 
\end{equation}
where $s_p(n)=\alpha_0+\alpha_1+\ldots+\alpha_r$ is  the sum of
the digits in the $p$-adic expansion  $n=\alpha_0+\alpha_1p+\ldots+\alpha_rp^r$.
Then $n!\,d_n$ is a valid common denominator for all coefficients of words
of length $n$ in the Baker--Campbell--Hausdorff series $H=\log(\ee^{\AA}\ee^{\BB})$, or, equivalently,\footnote{Formally, we define $\denom(r)$
for $r\in\mathbb{Q}$ as the smallest
positive integer $d$ such that $r\cdot d\in\mathbb{Z}$.
In particular, $\denom(0)=1$.}
\begin{equation*}
\mathrm{denom}(\coeff(w, H))\ | \ n!\,d_n,\quad w\in\{\AA,\BB\}^n.
\end{equation*}
\end{theorem}
\begin{remark}The theorem can be extended to the case of $K\geq 2$ exponentials,
\begin{equation*}
\mathrm{denom}(\coeff(w, \log(\ee^{\AA_1}\!\cdots\ee^{\AA_K})))\ | \ n!\,d_n,\quad w\in\{\AA_1,\dots\AA_K\}^n.
\end{equation*} 
Our proof of Theorem~\ref{Thm:MainTheorem}  
in Section~\ref{Sec:Proof} will cover this more general case as well.
\end{remark}
An explicit computation \cite{HHfast} yields\footnote{Here $\lcm\,\mathcal{M}$ denotes the {\em least common multiple} of the elements of the finite set $\mathcal{M}\subset\mathbb{Z}$.}
\begin{equation*}
\mathrm{lcm}\big\{\mathrm{denom}(\coeff(w,H)): w\in\{\AA,\BB\}^n
\big\}=n!\,d_n,\quad n=1,2,\dots,30,
\end{equation*}
which shows that at least for $n\leq 30$
the common denominators $n!\,d_n$ are as small as possible.
The first few values of $d_n$ are
\begin{equation*}
d_n = 1, 1, 2, 1, 6, 2, 6, 3, 10, 2, 6, 2, 210, 30, 12, 3, 30, 10, 210, 42, 330, 30, 60, 30, 546, \ldots. 
\end{equation*}
A search in the Online Encyclopedia of Integer Sequences \cite{SloaneOEIS} does not (yet)  result in
a match for $\{d_n\}$, but remarkably there is a near match, namely the sequence A195441,
\begin{equation*}
\widetilde{d}_n = 1, 1, 2, 1, 6, 2, 6, 3, 10, 2, 6, 2, 210, 30, 6, 3, 30, 10, 210, 42, 330, 30, 30, 30, 546, \dots,
\end{equation*}
which is  investigated in \cite{KellnerSondow}.\footnote{In \cite{KellnerSondow}
the sequence $\{\widetilde{d}_n\}$ is denoted $\{q_n\}$ and 
starts with index $n=0$ such that $\widetilde{d}_n = q_{n-1}$, $n=1,2,\dots$.}
For $n\leq 25$ we have $d_n\neq\widetilde{d}_n$ only for $n=15$ and  $n=23$.
As it turns out, 
$\widetilde{d}_n$ is the square-free kernel of $d_n$,
\begin{equation*}
\widetilde{d}_n = 
\prod_{p|d_n}p =
\prod_{p<n:\ s_p(n)\geq p} p,
\end{equation*}
and there is a connection with  the Bernoulli numbers and the Bernoulli polynomials,
\begin{equation*}
\widetilde{d}_n = \mathrm{denom}(B_n(x)-B_n), 
\end{equation*}
see \cite{KellnerSondow}.

The rest of the paper is organized as follows. First, still in this introduction, 
we prove two corollaries to Theorem~\ref{Thm:MainTheorem} which for 
special degrees $n$ provide information about the numerators
of the BCH coefficients $h_w$. Then we give an example in which 
the results of Theorem~\ref{Thm:MainTheorem} and Corollary~\ref{Cor:n_prime}
are verified by explicit computations. 
Our  proof of  Theorem~\ref{Thm:MainTheorem} is naturally divided into a combinatorial and a number-theoretical part. The combinatorial part is given
in Section~\ref{Sec:Preliminaries} and leads to a preliminary result in 
Proposition~\ref{Prop:D_n}. Based on this preliminary result we finally
prove Theorem~\ref{Thm:MainTheorem} in the number-theoretical part in
Section~\ref{Sec:Proof} .

\begin{corollary}\label{Cor:n_prime}
Let $p\geq2$ prime, and let $w\in\{\AA,\BB\}^p\setminus\{\AA^p, \BB^p\}$ be a word of length $p$ different from 
$\AA^p$ and $\BB^p$. If the coefficient  of $w$ in the Baker--Campbell--Hausdorff series 
$H=\log(\ee^{\AA}\ee^{\BB})$
 is written with denominator
$p!\,d_p$, 
\begin{equation*}
h_w=\coeff(w, H)=\frac{a_w}{p!\,d_p},
\end{equation*}
then the numerator $a_w\in\mathbb{Z}$  satisfies
\begin{equation*}
a_w \equiv - d_p \ (\mathrm{mod}\ p).
\end{equation*}
\end{corollary}
\begin{proof}
In \cite[Section~IV.A]{VanBruntVisser} it is shown that for $w\in\{\AA,\BB\}^p\setminus\{\AA^p, \BB^p\}$ the coefficient $h_w$  can be written as
\begin{equation*}
h_w = \frac{1}{p} + C_w,
\end{equation*}
where $C_w$ is a rational number whose denominator is not divisible by $p$. It follows
\begin{equation*}
a_w = \left(\frac{1}{p} + C_w\right)\, p!\,d_p 
\equiv (p-1)!\,d_p
\equiv -d_p \ (\mathrm{mod} \ p),
\end{equation*}
where in the last step we used Wilson's theorem.
\end{proof}
\begin{corollary}
Let $n=p+1$ with $p\geq 3$ an odd prime.
Define the exceptional set $$\mathcal{Z}_{p+1}=\{\AA\BB^p, \BB^p\AA, \AA^p\BB, \BB^p\AA\}\cup\{w=w_1\!\cdots w_{p+1}\in\{\AA,\BB\}^{p+1}: \ w_1=w_{p+1}\}.$$ 
Let $w\in\{\AA,\BB\}^{p+1}\setminus\mathcal{Z}_{p+1}$. 
If the coefficient  of $w$ in the Baker--Campbell--Hausdorff series  
$H=\log(\ee^{\AA}\ee^{\BB})$
is written with denominator
$(p+1)!\,d_{p+1}$, 
\begin{equation*}
h_w=\coeff(w, H)=\frac{a_w}{(p+1)!\,d_{p+1}},
\end{equation*}
then the numerator $a_w\in\mathbb{Z}$  satisfies
\begin{equation*}
a_w \equiv \frac{p-1}{2}\,d_{p+1} \ (\mathrm{mod}\ p).
\end{equation*}
If, on the other hand, $w\in\mathcal{Z}_{p+1}$, then  $\coeff(w,H)=0$.
\end{corollary}
\begin{proof}
In \cite[Section~IV.B]{VanBruntVisser} it is shown that 
$h_w=0$ for $w\in\mathcal{Z}_{p+1}$, and that 
for $w\in\{\AA,\BB\}^{p+1}\setminus\mathcal{Z}_{p+1}$ the coefficient $h_w$  can be written as
\begin{equation*}
h_w = \frac{1}{2p} + C_w,
\end{equation*}
where $C_w$ is a rational number whose denominator is not divisible by $p$. It follows
\begin{equation*}
a_w = \left(\frac{1}{2p} + C_w\right)\, (p+1)!\,d_{p+1} 
\equiv (p-1)!\frac{p+1}{2}\,d_{p+1}
\equiv \frac{p-1}{2}\,d_{p+1}
\ (\mathrm{mod} \ p).
\end{equation*}
\end{proof}

\begin{example}
\begin{table}[t ]
\centering
\begin{tabular}{rl@{\hskip 0pt}l@{\hskip 0pt}l@{\hskip 0pt}l@{\hskip 0pt}lr|rl@{\hskip 0pt}l@{\hskip 0pt}l@{\hskip 0pt}l@{\hskip 0pt}lr}
\hline
\multicolumn{1}{c}{$h_w$} &\multicolumn{5}{c}{$\mathrm{denom}(h_w)$}& \multicolumn{1}{c|}{$a_w$} &
\multicolumn{1}{c}{$h_w$} &\multicolumn{5}{c}{$\mathrm{denom}(h_w)$}& \multicolumn{1}{c}{$a_w$} \\
\hline\rule{0pt}{10pt}
$1/47900160$ &$2^9$ &$\cdot 3^5$ &$\cdot 5$   &$\cdot 7$ &$\cdot 11$ &$5$     & 
$1/739200$   &$2^7$ &$\cdot 3$   &$\cdot 5^2$ &$\cdot 7$ &$\cdot 11$ &$324$\\
$-1/4790016$ &$2^8$ &$\cdot 3^5$ &            &$\cdot 7$ &$\cdot 11$ &$-50$   &
$-13/554400$ &$2^5$ &$\cdot 3^2$ &$\cdot 5^2$ &$\cdot 7$ &$\cdot 11$ &$-5616$\\
$1/1064448$  &$2^9$ &$\cdot 3^3$ &            &$\cdot 7$ &$\cdot 11$ &$225$   &
$17/4435200$ &$2^8$ &$\cdot 3^2$ &$\cdot 5^2$ &$\cdot 7$ &$\cdot 11$ &$918$\\
$1/1247400$  &$2^3$ &$\cdot 3^4$ &$\cdot 5^2$ &$\cdot 7$ &$\cdot 11$ &$192$   &
$1/88704$    &$2^7$ &$\cdot 3^2$ &            &$\cdot 7$ &$\cdot 11$ &$2700$\\
$-1/399168$  &$2^6$ &$\cdot 3^4$ &            &$\cdot 7$ &$\cdot 11$ &$-600$  &
$-17/5322240$&$2^9$ &$\cdot 3^3$ &$\cdot 5$   &$\cdot 7$ &$\cdot 11$ &$-765$\\
$-13/6652800$&$2^7$ &$\cdot 3^3$ &$\cdot 5^2$ &$\cdot 7$ &$\cdot 11$ &$-468$  &
$1/332640$   &$2^5$ &$\cdot 3^3$ &$\cdot 5$   &$\cdot 7$ &$\cdot 11$ &$720$\\
$-1/277200$  &$2^4$ &$\cdot 3^2$ &$\cdot 5^2$ &$\cdot 7$ &$\cdot 11$ &$-864$  &
$1/3991680$  &$2^7$ &$\cdot 3^4$ &$\cdot 5$   &$\cdot 7$ &$\cdot 11$ &$60$\\
$-1/712800$  &$2^5$ &$\cdot 3^4$ &$\cdot 5^2$ &          &$\cdot 11$ &$-336$  &
$13/665280$  &$2^6$ &$\cdot 3^3$ &$\cdot 5$   &$\cdot 7$ &$\cdot 11$ &$4680$\\
$1/228096$   &$2^8$ &$\cdot 3^4$ &            &          &$\cdot 11$ &$1050$  &
$13/7983360$ &$2^8$ &$\cdot 3^4$ &$\cdot 5$   &$\cdot 7$ &$\cdot 11$ &$390$\\
$7/2851200$  &$2^7$ &$\cdot 3^4$ &$\cdot 5^2$ &          &$\cdot 11$ &$588$   &
$-1/124740$  &$2^2$ &$\cdot 3^4$ &$\cdot 5$   &$\cdot 7$ &$\cdot 11$ &$-1920$\\
$1/158400$   &$2^6$ &$\cdot 3^2$ &$\cdot 5^2$ &          &$\cdot 11$ &$1512$  &
$-1/33264$   &$2^4$ &$\cdot 3^3$ &            &$\cdot 7$ &$\cdot 11$ &$-7200$\\
$1/1900800$  &$2^8$ &$\cdot 3^3$ &$\cdot 5^2$ &          &$\cdot 11$ &$126$   &
$-1/10395$   &$3^3$ &$\cdot 5$   &            &$\cdot 7$ &$\cdot 11$ &$-23040$\\
$-1/190080$  &$2^7$ &$\cdot 3^3$ &$\cdot 5$   &          &$\cdot 11$ &$-1260$ &
$-1/73920$   &$2^6$ &$\cdot 3$   &$\cdot 5$   &$\cdot 7$ &$\cdot 11$ &$-3240$\\
$-1/887040$  &$2^8$ &$\cdot 3^2$ &$\cdot 5$   &$\cdot 7$ &$\cdot 11$ &$-270$  &
$1/27720$    &$2^3$ &$\cdot 3^2$ &$\cdot 5$   &$\cdot 7$ &$\cdot 11$ &$8640$\\
$-17/1663200$&$2^5$ &$\cdot 3^3$ &$\cdot 5^2$ &$\cdot 7$ &$\cdot 11$ &$-2448$ &
$-1/2772$    &$2^2$ &$\cdot 3^2$ &            &$\cdot 7$ &$\cdot 11$ &$-86400$\\
\hline
\end{tabular}
\caption{Possible values and factorizations of their denominators for the coefficients $h_w=\coeff(w, \log(\ee^\AA\ee^\BB))$ corresponding to words $w$
of length $n=11$. The $a_w$ are the numerators of these coefficients if written
with denominator $11!\,d_{11}=239500800$.}\label{tab:h11}
\end{table}
We consider the case $n=11$. The coefficients $h_w$ corresponding to the $2^{11}-2=2046$ words
$w\in\{\AA,\BB\}^{11}\setminus\{\AA^{11},\BB^{11}\}$
only take values from a set of 30 elements. These 30 possible values of
the coefficients are displayed in Table~\ref{tab:h11} and can be looked up in \cite{NewmanThompson}.\footnote{The fact that so many coefficients have the
same value is not a coincidence but a consequence of certain
symmetries satisfied by the coefficients, see \cite{G}.}
Also displayed are the prime factorizations of the denominators of the
coefficients. The smallest common denominator for all these coefficients
is given by the least common multiple of the denominators, which
using the factorizations is readily determined to be 
$2^9\cdot 3^5\cdot 5^2\cdot 7 \cdot 11=239500800$.
The computations
\begin{equation*}
\begin{array}{lll}
11=1\cdot 2^3+1\cdot 2^1+ 1\cdot 2^0, & s_2(11)=3, &\max\{t:2^t\leq s_2(11)\}=1, \\
11=1\cdot 3^2+2\cdot 3^0, & s_3(11)=3, &\max\{t:3^t\leq s_3(11)\}=1, \\
11=2\cdot 5^1+1\cdot 5^0, & s_5(11)=3, &\max\{t:5^t\leq s_5(11)\}=0, \\
11=1\cdot 7^1+4\cdot 5^0, & s_7(11)=5, &\max\{t:7^t\leq s_7(11)\}=0
\end{array}
\end{equation*}
result in  
$d_{11}=2\cdot 3=6$ for the value defined by (\ref{eq:d_n}). Together
with $11!=39916800$ this gives $11!\,d_{11}=239500800$ which 
is indeed the smallest possible common denominator.
Furthermore, Table~\ref{tab:h11} shows the numerators $a_w$ of the coefficients 
$h_w=a_w/(11!\,d_{11})$ written with denominator $11!\,d_{11}=239500800$.
Since $n=11$ is prime, we expect 
that $a_w\equiv-d_{11}=-6\equiv 5 \ (\mathrm{mod} \ 11)$ holds by Corollary~\ref{Cor:n_prime}.  
It is  readily verified that this is indeed the case, e.g., by 
computing the alternating digit sums 
of the numerators $a_w$, as in the well-known divisibility rule for $n=11$.
\end{example}

\section{A preliminary result}\label{Sec:Preliminaries}
If some information about the denominators of the coefficients of the sub-expressions $X_1,\dots,X_K\in\QQ\langle\langle\nA\rangle\rangle$ is available, one can expect that from it 
something can be learned about the denominators of the coefficients of the compound expressions $X_1+\ldots+X_K$ and $X_1\cdots X_K$. 
The following technical lemma makes this idea concrete.
We will apply this lemma to obtain
a preliminary result  about the denominators of the coefficients of
the BCH series $H=\log(\ee^\AA\ee^\BB)$ in Proposition~\ref{Prop:D_n}, which will be the starting
point for the proof of Theorem~\ref{Thm:MainTheorem} in Section~\ref{Sec:Proof}.
\begin{lemma}\label{Lemma:denomcoeff}
Let $n\geq 0$.
Let $X_i\in\QQ\langle\langle\nA\rangle\rangle$ 
and let $\delta_j(X_i)\in\mathbb{Z}_{>0}$ such that
 \begin{equation*}
\denom(\coeff(v,X_i))\ \big| \ \delta_j(X_i),\quad v\in\nA^j,\quad j=0,\dots,n,\
i=1,\dots,K.
\end{equation*} 
Then for $w\in\nA^n$, we have\\
(i)
\begin{equation*}
\denom(\coeff(w,\frac{a}{b}X_i))\ | \ \frac{b\,\delta_n(X_i)}{\gcd(b\,\delta_n(X_i),a)},
\quad a\in\mathbb{Z}, \ b\in\mathbb{Z}_{>0},
\end{equation*}
(ii)
\begin{equation*}
\denom(\coeff(w, X_1+\ldots+X_K))\ | \ \lcm\{\delta_n(X_1),\dots,\delta_n(X_K)\},
\end{equation*}
(iii)
\begin{equation*}
\denom(\coeff(w, X_1\cdots X_K))\ | \ \lcm\{\delta_{j_1}(X_1)\cdots\delta_{j_K}(X_K):\ j_i\geq 0,\ j_1+\ldots+j_K=n \}.
\end{equation*}
(iv) If the $X_i$ have no constant terms, $\coeff(1, X_i)=0$, $i=1,\dots,K$, then 
the last divisibility relation can be tightened to 
\begin{equation*}
\denom(\coeff(w, X_1\cdots X_K))\ | \ \lcm\{\delta_{j_1}(X_1)\cdots\delta_{j_K}(X_K):\ j_i\geq 1,\ j_1+\ldots+j_K=n \}.
\end{equation*}
\end{lemma}
\begin{proof} 
(i) follows from
\begin{equation*}
\denom\left(\frac{a}{b}\frac{c}{d}\right)\ | \ \frac{bd}{\gcd(bd, a)},\quad a,c\in\mathbb{Z},\
b,d\in\mathbb{Z}_{>0}.
\end{equation*}
(ii) follows from
\begin{equation}\label{eq:denom_sum}
\denom(r_1+\ldots+r_K)\ | \ \lcm\{\denom(r_1),\dots,\denom(r_K)\},\quad r_1,\dots,r_K\in\QQ.
\end{equation}
Ad (iii). By distributing the subwords $v^{(1)},\dots,v^{(K)}$
of all partitions $w=v^{(1)}\!\cdots v^{(K)}$ of $w$ into $K$ subwords 
of length $|v^{(i)}|\geq 0$ among 
the factors $X_1,\dots, X_K$ and summing over all such partitions we obtain
\begin{equation*}
\coeff(w, X_1\cdots X_K)=\sum_{v^{(1)}\cdots v^{(K)}=w}
\coeff(v^{(1)},X_1)
\cdots\coeff(v^{(K)},X_K).
\end{equation*}
Each partition  $w=v^{(1)}\cdots v^{(K)}$ into $K$ subwords 
uniquely corresponds 
to a partition $n=j_1+\ldots+j_K$ of $n=|w|$
into $K$ summands $j_1,\dots,j_K\geq 0$,  where the correspondence is 
given by
$(v^{(1)},\dots,v^{(K)})\mapsto(j_1,\dots,j_K)=(|v^{(1)}|,\dots,|v^{(K)}|)$. 
Using (\ref{eq:denom_sum}) and 
$$\denom(r_1\cdots r_K)\,|\,\denom(r_1)\cdots\denom(r_K),\quad r_1,\dots,r_K\in\QQ$$ 
it follows 
\begin{equation*}
\denom(\coeff(w, X_1\cdots X_K))\ | \ \lcm\{\delta_{j_1}(X_1)\cdots\delta_{j_K}(X_K):\ j_i\geq 0,\ j_1+\ldots+j_K=n \}.
\end{equation*}
The proof of (iv) is the same as the one of (iii) except that now 
only partitions $w=v^{(1)}\cdots v^{(K)}$ into subwords of
length $|v^{(1)}|\geq 1$ have to be considered.
Such partitions now correspond to partitions $n=j_1+\ldots+j_K$ into summands 
$j_1,\dots,j_K\geq 1$.
\end{proof}
\begin{proposition}\label{Prop:D_n}
Define
\begin{equation}\label{eq:D_n}
D_n = \lcm\{k\, j_1!\cdots j_k!:\ j_i\geq 1, \ j_1+\ldots+j_k=n,\ k=1,\dots,n\},\quad
n=1,2,\dots.
\end{equation}
Then 
\begin{equation*}
\denom(\coeff(w,\log(\ee^{\AA_1}\!\cdots\ee^{\AA_K})))\ | \  
D_n,\quad w\in\nA^n,\ \nA=\{\AA_1,\dots,\AA_k\}.
\end{equation*}
\end{proposition}
\begin{proof}
Because $\coeff(w, \ee^{\AA_i})\in\{0,1/n!\}$ for $w\in\nA^n$ it is clear that
\begin{equation*}
\denom(\coeff(w,\ee^{\AA_i}))\ | \ n!,\quad w\in\nA^n, \ i=1,\dots,K, \ n=0,1,2,\dots,
\end{equation*}
which
using Lemma~\ref{Lemma:denomcoeff}~(iii) implies
\begin{equation*}
\denom(\coeff(w,\ee^{\AA_1}\!\cdots\ee^{\AA_K}))\ | \
\lcm\{j_1!\cdots j_K!:\ j_i\geq0, \ j_1+\ldots+j_K=n\} \ = \ n!\,. 
\end{equation*}
We set
$Y = \ee^{\AA_1}\!\cdots\ee^{\AA_K}-1$
so that $\coeff(1, Y)=0$ and
\begin{equation*}
\denom(\coeff(w,Y))
\ | \ n!, \quad w\in\nA^n,
\ n=1,2,\dots,
\end{equation*}
which 
using Lemma~\ref{Lemma:denomcoeff}~(i),~(iv) implies 
\begin{equation*}
\denom\left(\coeff\left(w,\frac{(-1)^{k+1}}{k}Y^k\right)\right)
\ | \ \lcm\{k\,j_1!\cdots j_k!:\ j_i\geq 1, \ j_1+\ldots+j_k=n\}
\end{equation*}
for $k=1,\dots,n$ and $w\in\nA^n$.
Because $\coeff(w,Y^k)=0$ for $k>|w|$ we have
\begin{equation*}
\coeff(w,\log(1+Y))
=\coeff\left(w,\ \sum_{k=1}^{n}\frac{(-1)^{k+1}}{k}Y^k\right), \quad |w|\leq n,
\end{equation*}
from which 
\begin{equation*}
\denom(\coeff(w,\log(1+Y)))\, |\, 
\lcm\{k\,j_1!\cdots j_k!:\, j_i\geq 1, \, j_1+\ldots+j_k=n,\, k=1,\dots,n\}
\end{equation*}
follows for $w\in\nA^n$, $n=1,2,\dots$  by an application of  Lemma~\ref{Lemma:denomcoeff}~(ii).
\end{proof}
 \section{Proof of Theorem~\ref{Thm:MainTheorem}}\label{Sec:Proof}
For the proof of Theorem~\ref{Thm:MainTheorem} we will show that 
\begin{equation}\label{eq:d_eq_D}
 n!\,d_n = D_n,
\end{equation}
where $d_n$ is defined by (\ref{eq:d_n}) and $D_n$ is defined by (\ref{eq:D_n}).
Then, Theorem~\ref{Thm:MainTheorem} will be an immediate consequence of Proposition~\ref{Prop:D_n}.

For a prime $p\geq 2$  the  {\em $p$-adic valuation} $v_p(n)$ of $n$ is
defined as the exponent of the highest power of $p$ that divides $n$.
The function $v_p$ satisfies 
\begin{equation*}
v_p(n\cdot m) = v_p(n)+v_p(m),
\end{equation*}
which implies that (\ref{eq:d_eq_D}) is equivalent to 
\begin{equation*}
v_p(n!)+ v_p(d_n) = v_p(D_n) 
\quad \mbox{for all primes}\ p\geq 2,
\end{equation*}
where $v_p(d_n)=\max\{t:\ p^t \leq s_p(n)\}$.
Here and in the following $s_p(n)=\alpha_0+\ldots+\alpha_r$ is the sum of 
digits in the 
{\em $p$-adic expansion}  $n=\alpha_0+\alpha_1p+\dots+\alpha_rp^r$.
To compute $v_p(D_n)$ we need some further properties of the function
$v_p$. 

For nonempty finite subsets $\mathcal{M}\subset\mathbb{Z}_{\geq 0}$
we have
\begin{equation*}
v_p(\lcm\,\mathcal{M}) = \max_{m\in\mathcal{M}}v_p(m)
\end{equation*}
and, by convention, $\lcm(\emptyset)=1$ such that $v_p(\lcm(\emptyset))=0$.

For the computation of $v_p$ for factorials we have {\em Legendre's formula}
\begin{equation*}
v_p(n!)=\frac{n-s_p(n)}{p-1},
\end{equation*}
see, e.g., \cite{Mihet}.

Now, let $j_1,\dots,j_k\geq 1$ with $j_1+\ldots+j_k=n$. Then
\begin{align*}
v_p(k\,j_1!\cdots j_k!) &= 
v_p(k)+v_p(j_1!)+\ldots+v_p(j_k!)\\
&=v_p(k)+\frac{1}{p-1}\big(n-s_p(j_1)-\ldots-s_p(j_k)\big)\\
&=v_p(n!)+v_p(k)-\frac{1}{1-p}\big(s_p(j_1)+\ldots+s_p(j_k)-s_p(n)\big).
\end{align*}
It follows 
\begin{align}
v_p(D_n) &= \max_{k=1,\dots,n}\,\max_{j_i\geq 1,\, j_1+\ldots+j_k=n}v_p(k\,j_1!\cdots j_k!)\nonumber\\
&=v_p(n!)+\max_{k=1,\dots,n}\big((v_p(k)-h_p(n,k)\big), 
\label{eq:vp_D}
\end{align}
where
\begin{equation*}
h_p(n,k)=\frac{1}{p-1}\min_{j_i\geq 1,\, j_1+\ldots+j_k=n}\big(s_p(j_1)+\ldots+s_p(j_k)-s_p(n)\big).
\end{equation*}
To complete the computation of $v_p(D_n)$ we need
some properties of the function $h_p(n,k)$ which follow from the following
two lemmas.
\begin{lemma}\label{Lemma:min_h_1}
If $k\leq s_p(n)$, then
\begin{equation*}
\min_{j_i\geq 1,\ j_1+\ldots+j_k=n}\big(s_p(j_1)+\ldots+s_p(j_k)-s_p(n)\big)=0.
\end{equation*}
\end{lemma}
\begin{proof}
With the multinomial coefficient 
$\genfrac(){0pt}{}{n}{j_1,\dots,j_k}=\frac{n!}{j_1!\cdots j_k!}$
we have
\begin{equation}\label{eq:v_multinom}
\frac{1}{p-1}\big(s_p(j_1)+\ldots+s_p(j_k)-s_p(n)\big)=v_p
\left(
\genfrac(){0pt}{}{n}{j_1,\dots,j_k}
\right)\geq 0,
\end{equation}
and thus
\begin{equation*}
s_p(j_1)+\ldots+s_p(j_k)-s_p(n)\geq 0
\end{equation*}
for all $j_1,\dots, j_k\geq 1$ with $j_1+\ldots + j_k=n$.

Using the assumption $k\leq s_p(n)$ we now construct an assignment of the variables $j_1,\dots,j_k\geq 1$ for which 
$j_1+\ldots+j_k=n$ and $s_p(j_1)+\ldots+s_p(j_k)=s_p(n)$ hold.
 The existence of such
an assignment suffices to prove the lemma.

Corresponding to  the $p$-adic expansion
\begin{equation*}
n=\alpha_0+\alpha_1p+\dots+\alpha_rp^r
\end{equation*}
let $x$ ($0\leq x\leq r)$ be uniquely defined by the inequalities
\begin{equation*}
\alpha_0+\ldots+\alpha_{x-1}\leq k-1<
\alpha_0+\ldots+\alpha_x,
\end{equation*}
and let  $y$ ($0\leq y < \alpha_x$) be defined by by the equation
\begin{equation*}
k-1 = \alpha_0+\ldots+\alpha_{x-1}+y.
\end{equation*}
Note that here for the  existence of $x$ the requirement 
$k-1<s_p(n)=\alpha_0+\ldots+\alpha_r$ is necessary.
Define $j_1,\dots,j_{k-1}$ by
\begin{align*}
j_i=1,&\quad i=1,\,\dots,\,\alpha_0,\\
j_i=p,&\quad i=\alpha_0+1,\,\dots,\,\alpha_0+\alpha_1,\\
j_i=p^2,&\quad i=\alpha_0+\alpha_1+1,\,\dots,\,\alpha_0+\alpha_1+\alpha_2,\\
\vdots\quad\ \ &\quad \quad\vdots \\
j_i=p^{x-1},&\quad i=\alpha_0+\ldots+\alpha_{x-2}+1,\,\dots,\,\alpha_0+\ldots+\alpha_{x-2}+\alpha_{x-1}, \\
j_i=p^{x},&\quad i=\alpha_0+\ldots+\alpha_{x-1}+1,\,\dots,\,\alpha_0+\ldots+\alpha_{x-1}+y=k-1,
\end{align*}
and $j_k$ by
\begin{equation*}
j_k = (\alpha_x-y)p^x+\alpha_{x+1}p^{x+1}+\ldots+\alpha_rp^r.
\end{equation*}
Then $s_p(j_i)=1$, $i=1,\dots,k-1$ and $s_p(j_k)=(\alpha_x-y)+\alpha_{x-1}+\ldots+\alpha_r$, and thus
\begin{align*}
s_p(j_1)+\dots+s_p(j_{k-1})+s_p(j_k) 
&= k-1
\, +\, \alpha_x-y+\alpha_{x+1}+\ldots+\alpha_r \\
&= \alpha_0+\ldots+\alpha_{x-1}+y
\, + \, \alpha_x-y+\alpha_{x+1}+\ldots+\alpha_r \\
&= \alpha_0+\ldots+\alpha_r = s_p(n).
\end{align*}
Similarly, it is  easy to check that $j_1+\ldots+j_k=n$, and it is clear
that $j_i,\dots,j_k\geq 1$ (for $j_k$ this follows from $y<\alpha_x$).

\end{proof}

\begin{lemma}\label{Lemma:min_h_2}
For $n\geq  1$ let $l=\max\{t:\ p^t \leq s_p(n)\}$ such that $p^l\leq s_p(n)<p^{l+1}$, and let $k=p^{l+m}x> s_p(n)$ with 
$m\geq 1$ and $x\geq 1$.
Then
\begin{equation}\label{eq:geq_m}
\frac{1}{p-1}\big(s_p(j_1)+\ldots+s_p(j_k)-s_p(n)\big)\geq m
\end{equation}
for all $j_1,\dots, j_k\geq 1$ with $j_1+\ldots + j_k=n$.
\end{lemma}
\begin{proof}
We have
\begin{align*}
\frac{1}{p-1}\big(s_p(j_1)+\ldots+s_p(j_k)-s_p(n)\big)
&> \frac{1}{p-1}\big(p^{l+m}x-p^{l+1}\big)\\
&\geq \frac{p^{l+1}}{p-1}\big(p^{m-1}-1\big)\\
&=p^{l+1}\big(p^{m-2}+p^{m-3}+\ldots+1\big)\\
&\geq 2(m-1)\geq m-1.
\end{align*}
(Here, if $m=1$,  the sum $p^{m-2}+p^{m-3}+\ldots+1$
in the next-to-last row is unterstood to be $=0$.)
Since $\big(s_p(j_1)+\ldots+s_p(j_k)-s_p(n)\big)/(p-1)$ is an integer according to (\ref{eq:v_multinom}), this implies
(\ref{eq:geq_m}).
\end{proof}
We are now in the position to complete the computation (\ref{eq:vp_D}) of $v_p(D_n)$. 
Let $l=v_p(d_n)=\max\{t:\ p^t \leq s_p(n)\}$ such that $p^l \leq s_p(n)<p^{l+1}$.
For $k\in\{1,\dots,n\}$ we have the following 3 mutually exclusive possibilities: 
\begin{enumerate}[label=(\roman*), nosep]
\item If $k\leq s_p(n)$, then $v_p(k)=l$ and $h_p(n,k)=0$ by Lemma~\ref{Lemma:min_h_1}; thus $v_p(k)-h_p(n,k)=l$.
\item If $k=p^{l+m}x>s_p(n)$, $m\geq 1$, $x\geq 1$, $p\nmid x$, then $v_p(k)=l+m$ and
$h_p(n,k)\geq m$ by Lemma~\ref{Lemma:min_h_2}; thus $v_p(k)-h_p(n,k)\leq l$.
\item If $k=p^tx>s_p(n)$, $t\leq l$, $x\geq 1$, $p\nmid x$, then
$v_p(k)=t\leq l$ and $h_p(n,k)\geq 0$; thus $v_p(k)-h_p(n,k)\leq l$.
\end{enumerate}
Altogether this implies
\begin{equation*}
v_p(D_n) =v_p(n!)+\max_{k=1,\dots,n}\big((v_p(k)-h_p(n,k)\big)
=v_p(n!)+l= v_p(n!)+ v_p(d_n)
\end{equation*}
for all primes $p\geq 2$, which as already mentioned is equivalent to (\ref{eq:d_eq_D}), and thus completes the proof of Theorem~\ref{Thm:MainTheorem}.

%
%
%
%
%
%
%
%
%
%
%
%
%
%
%
%
%

\end{document}